\documentclass[12pt]{article}
\usepackage{e-jc}

\usepackage{amsthm,amsmath,amssymb}

\usepackage{graphicx}

\usepackage[colorlinks=true,citecolor=black,linkcolor=black,urlcolor=blue]{hyperref}

\theoremstyle{plain}
\newtheorem{theorem}{Theorem}
\newtheorem{lemma}[theorem]{Lemma}

\theoremstyle{definition}

\newtheorem{conjecture}[theorem]{Conjecture}

\theoremstyle{remark}

\title{\bf A Note on Weak Dirac Conjecture}

\author{Zeye Han\\
\small Kunming, Yunnan, China\\
\small\tt migaoyan@sina.com\\
}

\begin{document}

\maketitle


\begin{abstract}
  We show that every set $\mathcal{P}$ of $n$ non-collinear points in the plane contains a point incident to at least $\lceil\frac{n}{3}\rceil+1$ of the lines determined by $\mathcal{P}$.

\end{abstract}

\section{Introduction}

In the note we denote by $\mathcal{P}$ a set of non-collinear points in the plane, and $\mathcal{L}(\mathcal{P})$ the set of lines determined by $\mathcal{P}$, a line that passing through at least two points in $\mathcal{P}$ is said to be determined by $\mathcal{P}$. The system of lines spanned by a finite point set in the plane usually called configuration or arrangement. For a point $P$ of $\mathcal{P}$, we shall denote by $d(P)$ the number of lines of $\mathcal{L}(\mathcal{P})$ which are incident to $P$, and $d(P)$ is usually called the incident-line-number or multiplicity of $P$, see \cite{BMP} and \cite{KW} for details. And, we shall denote by $l_r$ the number of $r$-rich lines, that is they passing through precisely $r$ points of $\mathcal{P}$.

The Dirac's conjecture is a well-known problem in combinatorial
geometry. It was proven by Dirac from sixty years earlier. In 1951, Dirac (\cite{Dirac}) showed that,
\begin{theorem}
Every set $\mathcal{P}$ of $n$ non-collinear points in the plane contains a point incident to at least $\sqrt{n}$ lines of $\mathcal{L}(\mathcal{P})$.
\end{theorem}

Dirac (\cite{Dirac}) also made the following conjecture and checked the truth for $n\leq14$. In the case it is easy to see that this is best-possible, if $\mathcal{P}$ is equally distributed on two lines which $n$ is even, the situation is similar when $n$ is odd, then this bound is tight.
\begin{conjecture}[\textbf{Dirac's Conjecture}]
Every set $\mathcal{P}$ of $n$ non-collinear points in the plane contains a point incident to at least $\lfloor\frac{n}{2}\rfloor$ lines of $\mathcal{L}(\mathcal{P})$, where $\lfloor x\rfloor$ denotes the largest integer not exceeded by $x$.
\end{conjecture}

This conjecture is certainly false for small $n$, some counter-examples were found by Gr\"{u}nbaum for relatively small values of $n$, namely $n\in\{15, 19, 25, 31, 37\}$, { see \cite{Grunbaum1} and \cite{Grunbaum2} for details. In 2011, Akiyama, Ito, Kobayashi, and Nakamura (\cite{AIKN}) constructed a set $\mathcal{P}$ of non-colinear $n$ points for every integer $n\geq8$ except $n=12k+23$, $k\geq0$, satisfying $d(P)\leq\lfloor\frac{n}{2}\rfloor$ for every point $P$ of $\mathcal{P}$.

Then the following conjecture is naturally proposed, see \cite{BMP} for details.

\begin{conjecture}
[\textbf{Strong Dirac Conjecture}] Every set $\mathcal{P}$ of $n$ non-collinear points in the plane contains a point incident to at least $\frac{n}{2}-c_0$ lines of $\mathcal{L}(\mathcal{P})$, for some constant $c_0$.
\end{conjecture}

In 1961, Erd\"{o}s (\cite{Erdos}) proposed the following weakened conjecture.
\begin{conjecture}[\textbf{Weak Dirac Conjecture}]
Every set $\mathcal{P}$ of $n$ non-collinear points contains a point incident to at least $\frac{n}{c_1}$ lines of $\mathcal{L}(\mathcal{P})$, for some constant $c_1$.
\end{conjecture}

In 1983, the Weak Dirac Conjecture was proved independently by Beck \cite{Beck} and Szemer\'{e}di and Trotter \cite{ST} for the case with $c_1$ unspecified or very large.

In 2012, based on Crossing Lemma, Szemer\'{e}di-Trotter Theorem and Hirzebruch's Inequality, Payne and Wood (\cite{PW}) proved the
following theorem,
\begin{theorem}
Every set $\mathcal{P}$ of $n$ non-collinear points contains a point incident to at least $\frac{n}{37}$ lines of $\mathcal{L}(\mathcal{P})$.
\end{theorem}

In 2016, Pham and Phi (\cite{PP}) refined the result of Payne and Wood to given that,
\begin{theorem}
Every set $\mathcal{P}$ of $n$ non-collinear points contains a point incident to at least $\frac{n}{26}+2$ lines of $\mathcal{L}(\mathcal{P})$.
\end{theorem}

There are some results from algebraic geometry involving the system of lines spanned by a finite set in the plane. In 1983, Hirzebruch studied arrangements in the complex projective plane by associating with each arrangement some algebraic surfaces, calculating their Chern numbers, and then applying the Bogomolov-Miyaoka-Yau Inequality. In 1984, Hirzebruch \cite{Hirzebruch} obtained the following result,

\begin{theorem}[\textbf{Hirzebruch's Inequality}]
Let $\mathcal{P}$ be a set of $n$ points in the plane with at most $n-3$ collinear, then $$l_2+\frac{3}{4}l_3\geq n+\sum_{r\geq 5}(2r-9)l_r.$$
\end{theorem}

In 2003, the following Hirzebruch-type inequality was announced in the thesis \cite{Bojanowski} as Lemma 2.2 based on Langer's work \cite{Langer}.

\begin{theorem}
  \label{Thm:FabGraphs0}
  Let $\mathcal{P}$ be a set of $n$ points in the plane with at most
 $\lfloor\frac{2}{3}n\rfloor$ collinear, then
 $$l_2+\frac{3}{4}l_3\geq n+\sum_{r\geq 5}(\frac{r^2}{4}-r)l_r.$$
\end{theorem}

For the Weak Dirac Conjecture, based on results of \cite{Langer} and \cite{Bojanowski}, we show that,

\begin{theorem}
  \label{Thm:FabGraphs1}
  Every set $\mathcal{P}$ of $n$  non-collinear points in the plane contains a point incident to at least $\lceil\frac{n}{3}\rceil+1$ lines of $\mathcal{L}(\mathcal{P})$, where $\lceil x\rceil$ denotes the smallest integer greater than or equal to $x$.
\end{theorem}

\section{Proof of Theorem~\ref{Thm:FabGraphs1}}

By double-counting pairs of points in $\mathcal{P}$, we have the following equation,

\begin{lemma}
  \label{lem:Technical1}
  Let $\mathcal{P}$ a set of $n$ points in the plane, and $\mathcal{L}(\mathcal{P})$ the set of lines determined by $\mathcal{P}$, $l_r$ the number of $r$-rich lines of $\mathcal{L}(\mathcal{P})$, then
  $\sum_{r\geq2}{r\choose2}l_r={n\choose2}$.
\end{lemma}

\begin{lemma}
  \label{lem:Technical2}
Let $\mathcal{P}$ a set of $n$ points in the plane, and $\mathcal{L}(\mathcal{P})$ the set of lines determined by $\mathcal{P}$, $d(P)$ the incident-line-number of $P$, and $l_r$ the number of $r$-rich lines of $\mathcal{L}(\mathcal{P})$, then $$\sum_{P\in\mathcal{P}}d(P)=\sum_{r\geq2}rl_r.$$
\end{lemma}

\begin{proof}
Summing the incident-line-numbers counts each $r$-rich line $r$ times, since each $r$-rich line contains $r$ points and contributes to the incident-line-number at every point.

\end{proof}

\begin{proof}[Proof of Theorem~\ref{Thm:FabGraphs1}]
Note that if $\mathcal{P}$ is non-collinear and contains $\lceil\frac{n}{c_1}\rceil+1$ or more collinear points, then Theorem~\ref{Thm:FabGraphs1} holds, thus we may assume that $\mathcal{P}$ does not contain $\lceil\frac{n}{3}\rceil+1$ collinear points. According to Theorem~\ref{Thm:FabGraphs0}, we have $$l_2+\frac{3}{4}l_3\geq n+\sum_{r\geq 5}(\frac{r^2}{4}-r)l_r\Leftrightarrow l_2+\frac{3}{4}l_3\geq n+\sum_{r\geq 5}\frac{{r\choose2}}{2}l_r-\frac{3}{4}\sum_{r\geq 5}rl_r.$$

Applying the Lemma~\ref{lem:Technical1}, we have, $$l_2+\frac{3}{4}l_3\geq n+\frac{{n\choose2}}{2}-\sum_{r= 2}^4\frac{{r\choose2}}{2}l_r-\frac{3}{4}\sum_{r\geq 5}rl_r\Leftrightarrow\sum_{r\geq2}rl_r\geq\frac{n(n+3)}{3}.$$

Applying the Lemma~\ref{lem:Technical2}, we have, $$\sum_{P\in\mathcal{P}}d(P)\geq\frac{n(n+3)}{3}.$$

Finally, according to the Pigeonhole Principle \cite{Pigeonhole}, this implies that every set $\mathcal{P}$ of $n$ points does not contain $\lceil\frac{n}{3}\rceil+1$ collinear points in the plane contains a point incident to at least $\lceil\frac{n}{3}\rceil+1$ lines of $\mathcal{L}(\mathcal{P})$.

Therefore every set $\mathcal{P}$ of $n$ non-collinear points in the plane contains a point incident to at least $\lceil\frac{n}{3}\rceil+1$ lines of $\mathcal{L}(\mathcal{P})$.

\end{proof}

\subsection*{Acknowledgements}
I am very grateful to Professor David Wood and other reviewers for their useful comments about the content of the note, they are very fast and give me some encouragement. I would like to thank Professor Piotr Pokora because I know Theorem~\ref{Thm:FabGraphs0} via his paper.


\end{document}